\documentclass[11pt]{article}
\usepackage{amsfonts,amsmath,amsthm,amscd,amssymb,latexsym,cite,verbatim,texdraw,floatflt,
caption2,pb-diagram}

\theoremstyle{plain}
\newtheorem{theorem}{Theorem}[section]
\newtheorem{lemma}{Lemma}[section]

\newtheorem{corollary}{Corollary}[section]
\newtheorem{definition}{Definition}[section]
\theoremstyle{definition}
\newtheorem{example}{Example}[section]
\newtheorem{remark}{Remark}[section]
\newcommand{\keywords}{\textbf{Key words. }\medskip}
\newcommand{\subjclass}{\textbf{MSC 2020. }\medskip}
\renewcommand{\abstract}{\textbf{Abstract. }\medskip}

\numberwithin{equation}{section}

\newcommand{\RR}{\mathbb{R}}
\newcommand{\NN}{\mathbb{N}}

\begin{document}

\title{Fixed point theorems for weak, partial, Bianchini and Chatterjea-Bianchini contractions in semimetric spaces with triangle functions}

\author{Ravindra K. Bisht and Evgen O. Petrov}

\date{}

\maketitle

\begin{abstract}
This paper advances a line of research in fixed point theory initiated by M. Bessenyei and Z. P\'ales, building on their introduction of the triangle function concept in [J. Nonlinear Convex Anal, Vol 18 (3), 515-524 (2017)]. By applying this concept, the study revises several well-known fixed point theorems in metric spaces, extending their applicability to semimetric spaces with triangle functions. The paper focuses on general theorems involving weak, partial, Bianchini and Chatterjea-Bianchini contractions,  deriving corollaries relevant to metric spaces, $b$-metric spaces, ultrametric spaces, and distance spaces with power triangle functions. Notably, several new and interesting findings emerge in the context of weak and partial contractions.
\end{abstract}




\section{Introduction}

Fixed point theorems, in general, provide a sturdy framework for understanding and tackling solutions to both linear and nonlinear problems encountered in biological, engineering, and physical sciences.

The concept of successive approximations for solving differential and integral equations and in approximation theory has roots in the works of Chebyshev, Picard, Caccioppoli, and others. However, in 1922, Banach~\cite{Ba22} was the first to correctly formulate the contraction mapping principle, known as the Banach Contraction Principle, in an abstract form suitable for diverse applications.

In 1968, Kannan's fixed point theorem proved a significant result independent of the Banach Contraction Principle \cite{Ka68}. This breakthrough spurred the development of numerous contractive definitions, many of which permitted discontinuities within their domains while characterizing metric completeness. Notable among these contractive conditions are those investigated by Chatterjee \cite{Ch72}, Ćirić, Reich, and Rus \cite{CB71, Re71, Ru71}, Zamfirescu \cite{Z72}, and Hardy-Rogers \cite{HR73}. These conditions share similar characteristics, thereby enriching the understanding of contractive mappings in metric spaces. For a comprehensive overview of various contractive definitions, we refer authors to the survey paper by Rhoades \cite{Rh77}.

Over more than a century, researchers worldwide have maintained a strong interest in fixed point theorems. This is evidenced by the publication of numerous articles and monographs in recent decades dedicated to fixed point theory and its applications. For comprehensive surveys of fixed point results and their diverse applications, refer to the monographs
~\cite{KK2001,AJS18,Su18,AKOR15,GK90}.

Let $X$ be a nonempty set. A mapping $d\colon X\times X\to \mathbb{R}^+$, where $\mathbb{R}^+=[0,\infty)$, is defined as a \emph{metric} if it satisfies the following axioms for all $x,y,z \in X$:
\begin{itemize}
  \item [(i)] $(d(x,y)=0)\Leftrightarrow (x=y)$ (non-negativity condition),
  \item [(ii)] $d(x,y)=d(y,x)$ (symmetric condition),
  \item [(iii)] $d(x,y)\leqslant d(x,z)+d(z,y)$ (triangle inequality).
\end{itemize}

The combination of a set $X$ with a distance-measuring function $d$ is termed a ``metric space''. If $d$ satisfies the first two axioms, namely (i) and (ii), it earns the label ``semimetric''. Consequently, when $d$ qualifies as a semimetric on $X$, the pair $(X,d)$ is referred to as a ``semimetric space''. These spaces were initially explored by Fr\'{e}chet in~\cite{Fr06}, where he referred to them as ``classes (E)''. Subsequently, they garnered attention from researchers, as evidenced by works such as~\cite{Ch17,Ni27,Wi31,Fr37,DP13}. In semimetric spaces, one can introduce the concepts of convergent and Cauchy sequences, as well as completeness, following the standard procedures.

The concept of $b-$metric spaces, originally introduced by Bakhtin~\cite{B89} under the name of quasi-metric spaces.  Czerwik later expanded on this concept by applying it to generalizations of Banach's fixed point theorem \cite{C98,C93}. In a $b-$metric space, the traditional triangle inequality is modified to include a constant \( K \geq 1 \), ensuring that for all \( x, y, z \in X \), the inequality
\[
d(x, y) \leq K [d(x, z) + d(z, y)]
\]
holds. Fagin and Stockmeyer \cite{FKS03} delved deeper into the relaxed form of the triangle inequality, terming it nonlinear elastic matching (NEM). They highlighted its diverse applications, including trademark shape analysis  \cite{CMV94}  and the measurement of ice floes  \cite{Mc91}. For further details one may see \cite{FS98},\cite{X09}.

An \emph{ultrametric} is a special type of metric that satisfies the strong triangle inequality:
$$
d(x, y) \leqslant \max \{d(x, z), d(z, y)\}.
$$
When this condition is met, the pair $(X, d)$ is known as an \emph{ultrametric space}. It is worth noting that Hausdorff formulated the ultrametric inequality in 1934, and  Krasner introduced the concept of ultrametric spaces in 1944\cite{Kr44}.

In 2017, Bessenyei and  P\'ales~\cite{BP17} extended the Matkowski fixed point theorem~\cite{Ma75} by introducing a definition of a triangle function $\Phi \colon \overline{\mathbb{R}}_{+}^2\to \overline{\mathbb{R}}^+$ for a semimetric $d$.  In this paper, we adopt this definition in a slightly different form, restricting the domain and the range of $\Phi$ by ${\RR}_+^2$ and ${\RR}^+$, respectively.

\begin{definition}\label{d21}
Consider a semimetric space $(X, d)$. We say that $\Phi \colon {\RR}^+\times{\RR}^+ \to {\RR}^+$ is a \emph{triangle function} for $d$ if $\Phi$ is symmetric and non-decreasing in both of its arguments, satisfies $\Phi(0,0)=0$ and, for all $x, y, z \in X$, the generalized triangle inequality
\begin{equation}\label{tr}
d(x,y)\leqslant \Phi(d(x,z), d(z,y))
\end{equation}
holds.
\end{definition}

Clearly, metric spaces, ultrametric spaces, and $b$-metric spaces are all examples of semimetric spaces, each with a specific triangle function: $\Phi(u,v)=u+v$ for metric spaces, $\Phi(u,v)=\max\{u,v\}$ for ultrametric spaces, and $\Phi(u,v)=K(u+v)$, $K\geqslant 1$ for $b$-metric spaces.

In~\cite{BP17}, the study focused on regular semimetric spaces with basic triangle functions that are continuous at the origin. These spaces were shown to be Hausdorff, with unique limits for convergent sequences, and to satisfy the Cauchy property. For further developments on semimetric spaces with generalized triangle inequalities, see~\cite{VH17,CJT18,BP22,KP22,PS22}. 

In~\cite{PSB24} generalizations of Banach, Kannan, Chatterjea, \'Ciri\'c-Reich-Rus fixed point theorems, as well as of the fixed point theorem for mappings contracting perimeters of triangles were proved via considering corresponding mappings in semimetric spaces with triangle functions.

In this paper, we extend the study of fixed point theorems by modifying assumptions to broaden their applicability within semimetric spaces using triangle functions. Our findings offer new insights, applicable not only to metric spaces but also to broader contexts, such as $b$-metric spaces, ultrametric spaces, and distance spaces with power triangle functions.

\section{Main results}

The following lemma and theorem establish necessary conditions imposed on triangle function ensuring that the semimetric is continuous.
\begin{lemma}\label{lq1}
Let $(a_n)$ and $(b_n)$ be sequences such that $a_n, b_n \geqslant 0$, $a_n, b_n \to 0$ as $n \to \infty$, and let $(c_n)$ be any sequence with $c_n \geqslant 0$. Suppose the following conditions hold for the triangle function $\Phi$:
\begin{equation}\label{q1}
l \leqslant \Phi(b_n, \Phi(a_n, c_n)),
\end{equation}
\begin{equation}\label{q2}
c_n \leqslant \Phi(a_n, \Phi(l, b_n)),
\end{equation}
where $l > 0$ is a positive real number, and
\begin{equation}\label{q3}
|\Phi(x_n, y_n) - y_n| \to 0 \text{ as } n \to \infty,
\end{equation}
for every sequence $(x_n)$ such that $x_n \to 0$ as $n \to \infty$ and every arbitrary sequence $(y_n)$ with $x_n, y_n \geqslant 0$. Then $c_n \to l$ as $n \to \infty$.
\end{lemma}

\begin{proof}
Let $s_n = \Phi(b_n, \Phi(a_n, c_n))$ and $t_n = \Phi(a_n, \Phi(l, b_n))$. By condition~(\ref{q3}), we have

\begin{equation}\label{e56}
\begin{split}
|s_n - c_n|=|\Phi(b_n, \Phi(a_n, c_n))-\Phi(a_n,c_n)+\Phi(a_n,c_n)-c_n|
\\
\leqslant
|\Phi(b_n, \Phi(a_n, c_n))-\Phi(a_n,c_n)|+|\Phi(a_n,c_n)-c_n| \to 0
\end{split}
\end{equation}
 and
$$
|t_n - l| = |\Phi(a_n, \Phi(l, b_n))-\Phi(l, b_n)+\Phi(l, b_n)-l|
$$
$$
\leqslant |\Phi(a_n, \Phi(l, b_n))-\Phi(l, b_n)|+|\Phi(l, b_n)-l|
\to 0 \text{ as } n \to \infty.
$$
Suppose $c_n \nrightarrow l$. Then by~(\ref{e56}) $s_n \nrightarrow l$. Since $s_n \geqslant l$ for every $n$ and $s_n \nrightarrow l$, it follows from~(\ref{e56}) that
$$
\limsup_{n \to \infty} s_n = \limsup_{n \to \infty} c_n > l.
$$
Given that $t_n \to l$, condition~(\ref{q2}) implies that $\limsup_{n \to \infty} c_n \leqslant l$, leading to a contradiction.
\end{proof}

\begin{theorem}\label{t13}
Let $(X,d)$ be a semimetric space with the triangle function $\Phi$ satisfying condition~(\ref{q3}). Then the semimetric $d$ is continuous.
\end{theorem}

\begin{proof}
Let $x, y \in X$ and let $(x_n)$ and $(y_n)$ be sequences in $X$ such that $x_n \to x$ and $y_n \to y$ as $n \to \infty$. Define $a_n = d(x_n, x)$, $b_n = d(y_n, y)$, $c_n = d(x_n, y_n)$, and $l = d(x, y)$. Clearly, conditions~(\ref{q1}) and~(\ref{q2}) are satisfied. By applying Lemma~\ref{lq1}, we obtain $d(x_n, y_n) \to d(x, y)$, which is the desired result.
\end{proof}

\begin{example}\label{ex1}
Let $\Phi(u, v) = (u^q + v^q)^{\frac{1}{q}}$, where $q > 0$. Consider any sequence $(x_n)$ such that $x_n \to 0$ as $n \to \infty$ and an arbitrary sequence $(y_n)$ with $x_n, y_n \geqslant 0$. We have
$$
\lim_{n \to \infty} \left|(x_n^q + y_n^q)^{\frac{1}{q}} - y_n\right| = 0.
$$
Thus, by Theorem~\ref{t13}, the semimetric $d$ is continuous when $\Phi(u, v) = (u^q + v^q)^{\frac{1}{q}}$, $q > 0$.
\end{example}

\begin{remark}\label{r25}
Let $(X, d)$ be a semimetric space with the power triangle function $\Phi(u, v) = (u^q + v^q)^{\frac{1}{q}}$, where $q > 0$. If $q \geqslant 1$, then $(X, d)$ is a metric space. Indeed, by~(\ref{tr}), it is sufficient to note that for such $q$, the inequality $(u^q + v^q)^{\frac{1}{q}} \leqslant u + v$ holds, as shown, e.g., in Section 2.12 of~\cite{HLP52}. If $0 < q < 1$, then $(X, d)$ is not necessarily a metric space. For instance, let $X = \{x, y, z\}$ with $d(x, y) = d(x, z) = 1$ and $d(y, z) = 3$. In this case, $(X, d)$ is clearly not a metric space. A simple verification of the triangle inequalities for $q = \frac{1}{2}$ can demonstrate this.

Distance spaces with power triangle functions were considered in~\cite{Gr16}.
\end{remark}

The following lemma, which we need below, was proved in~\cite{PSB24}. For the convenience of the reader, we give it with the proof.

\begin{lemma}\label{lem}
Let $(X,d)$ be a semimetric space with the triangle function $\Phi$
satisfying the following conditions:
\begin{itemize}
\item[1)] The equality
\begin{equation}\label{ee1}
  \Phi(ku,kv) = k\Phi(u,v)
\end{equation}
holds for all $k,u,v\in \RR^+$.
\item[2)] For every $0\leqslant \alpha <1$ there exists $C(\alpha)>0$ such that for every
$p\in \NN^+$ the inequality
\begin{equation}\label{ee2}
  \Phi(1,\Phi(\alpha, \Phi(\alpha^2,....,\Phi(\alpha^{p-1},\alpha^{p}))))\leqslant C(\alpha)
\end{equation}
holds.
\end{itemize}

Let $(x_n)$, $n=0,1,\ldots$, be a sequence in $X$ having the property that there exists $\alpha\in [0,1)$ such that
\begin{equation}\label{e31n0}
d(x_n,x_{n+1})\leqslant \alpha d(x_{n-1},x_{n})
\end{equation}
for all $n\geqslant 1$. Then $(x_n)$ is a Cauchy sequence.
\end{lemma}

\begin{proof}
By~(\ref{e31n0}) we have
$$d(x_1,x_2)\leqslant \alpha d(x_0,x_1), \,\,
d(x_2,x_3)\leqslant \alpha d(x_1,x_2), \,\,
d(x_3,x_4)\leqslant \alpha d(x_2,x_3), \,\,\ldots.$$
Hence, we obtain
\begin{equation}\label{e31n}
d(x_n,x_{n+1})\leqslant \alpha^n d(x_0,x_1).
\end{equation}

Applying consecutively the generalized triangle inequality~(\ref{tr}) to the points $x_n$, $x_{n+1}$, $x_{n+2}$, \ldots ,$x_{n+p}$, where $p\in \NN^+$, $p\geqslant 2$, we obtain
$$
d(x_n,\,x_{n+p})\leqslant \Phi(d(x_{n},\,x_{n+1}), d(x_{n+1},\,x_{n+p}))
$$
$$
\leqslant \Phi(d(x_{n},\,x_{n+1}),  \Phi(d(x_{n+1},\,x_{n+2}), d(x_{n+2},\,x_{n+p})))
$$
$$
\ldots
$$
\begin{equation*}
\leqslant \Phi(d(x_{n},\,x_{n+1}),\Phi(d(x_{n+1},\,x_{n+2}),\ldots, \Phi(d(x_{n+p-2},\,x_{n+p-1}), d(x_{n+p-1},\,x_{n+p}))))
\end{equation*}
(by the monotonicity of $\Phi$ and by inequalities~(\ref{e31n}) we have)
$$
\leqslant \Phi(\alpha^{n}d(x_0,x_1),\Phi(\alpha^{n+1}d(x_0,x_1),\cdots,\Phi(\alpha^{n+p-2}d(x_0,x_1),\alpha^{n+p-1}d(x_0,x_1)))) \leqslant
$$
(applying several times equality~(\ref{ee1}), we obtain)
$$
\leqslant
\alpha^{n}\Phi(1,\Phi(\alpha,\cdots,\Phi(\alpha^{p-2},\alpha^{p-1})))d(x_0,x_1).
$$
By condition~(\ref{ee2}) we obtain
\begin{equation*}
d(x_n,\,x_{n+p})\leqslant \alpha^{n}C(\alpha)d(x_0,x_1).
\end{equation*}
Since $0\leqslant\alpha<1$, we have $d(x_n,\,x_{n+p})\to 0$ as $n\to \infty$ for every $p\geqslant 2$. If $p=1$, then the relation $d(x_n,\,x_{n+1})\to 0$ follows from~(\ref{e31n}). Thus, $(x_n)$ is a Cauchy sequence, which completes the proof.
\end{proof}

\begin{remark}[\!\!\cite{PSB24}]
Let $(X,d)$ be a complete semimetric space. Then the sequence $(x_n)$ has a limit $x^*$. If additionally the semimetric $d$ is continuous, then we get
 $d(x_n,x_{n+p}) \to d(x_n,x^*)$ as $p \to \infty$. Also,  letting $p\to \infty$,  we get
\begin{equation}\label{ee5}
d(x_n,\,x^*)\leqslant \alpha^{n}C(\alpha)d(x_0,x_1).
\end{equation}
\end{remark}


\subsection{Partial contractions in semimetric spaces}

In \cite{N24}, the concept of partial contractivity was introduced in the context of a $b$-metric space. Let $X$ be a $b-$metric space. A self-map $T : X \to X$ is a partial contractivity on a $b-$metric space $(X,d)$ if there exist $\alpha, \beta \in \mathbb{R}$ with $0 \leq \alpha < 1$ and $\beta \geq 0$ such that for any $x, y \in X$,
\begin{equation} \label{pc}
d(Tx, Ty) \leq \alpha d(x, y) + \beta d(x, Tx).
 \end{equation}
Due to the symmetric property in a metric space, $T$ also satisfies the following dual condition:
\begin{equation} \label{pc1}
d(Tx, Ty) \leq \alpha d(x, y) + \beta d(y, Ty).
\end{equation}

If $(X,d)$ is a metric space and  $T:X \to X$ satisfies either a Banach contraction, Kannan contraction, Chatterjea contraction, Zamfirescu contraction, Hardy-Rogers contraction, or \'Ciri\'c-Reich-Rus contraction, then $T$ is a partial contraction \cite{N24}.

\begin{remark}
It is pertinent to mention here that fixed point theorems in a metric space $(X,d)$, satisfying a combination of conditions (\ref{pc}) and (\ref{pc1}) were initiated by Proinov in \cite{P06}.
\end{remark}

We begin with the following fixed point theorem:

\begin{theorem}\label{t2}
Let $(X,d)$ be a complete semimetric space with the triangle function $\Phi$ continuous at $(0,0)$ and satisfying conditions~(\ref{ee1}) and~(\ref{ee2}).
Let $T \colon X \to X$ satisfy~(\ref{pc}) with $0 \leq \alpha + \beta < 1$. Then $T$ has a unique fixed point.
\end{theorem}
\begin{proof}
Let $x_0\in X$. Define $x_n = Tx_{n-1}$, i.e, $x_n=T^n x_0$ for $n = 1, 2, \ldots$. Then, in view of (\ref{pc}), it follows straightforwardly that

\begin{align*}
    d(x_{n},x_{n+{1}})=d(Tx_{n-1},Tx_{n}) &\leqslant  \alpha d(x_{n-1}, x_{n}) + \beta d(x_{n-1},Tx_{n-1})  \\
               &=  (\alpha+\beta) d(x_{n-1}, x_{n}).
\end{align*}
Hence,
$$
d(x_{n},x_{n+{1}}) \leqslant \delta d(x_{n-1},x_{n}),
$$
where $\delta = \alpha+\beta$, $0\leqslant \delta <1$.

By Lemma~\ref{lem}, $(x_n)$ is a Cauchy sequence and  by the completeness of $(X,d)$, this sequence has a limit $x^*\in X$. Let us prove that $Tx^*=x^*$.  Using the generalized triangle inequality~(\ref{tr}), the continuity of $\Phi$ and (\ref{pc}), we have

\begin{align*}
d(x^*, Tx^*) &\leqslant \Phi\left(d(x^*, T^n x_0), d(T^n x_0, Tx^*)\right) \\
&= \Phi\left(d(x^*, T^n x_0), \alpha d(T^{n-1} x_0, x^*) + \beta d(T^{n-1} x_0, T^n x_0)\right) \to 0.
\end{align*}
This implies that $x^*$ is a fixed point.

Suppose that there exist two distinct fixed points $x$ and $y$. Then $Tx=x$ and $Ty=y$, which contradicts~(\ref{pc}) since $\alpha<1$.
\end{proof}

In~\cite{PSB24}, it was proved that the triangle functions
$\Phi(u,v)=u+v$, $\Phi(u,v)=\max\{u,v\}$, $\Phi(u,v)=(u^q+v^q)^{\frac{1}{q}}$, $q>0$, satisfy condition~(\ref{ee1}) and condition~(\ref{ee2}) with
$C(\alpha)$ being $1/(1-\alpha)$, $1$ and $1/(1-\alpha^q)^{\frac{1}{q}}$, respectively. Clearly, all these triangle functions are continuous at $(0,0)$. Hence, we have the following.
\begin{corollary}\label{c23}
Theorem~\ref{t2} holds for semimetric spaces with all triangle functions mentioned above. Moreover, the following estimations hold:
  $$
  d(x_n,x^*)\leqslant \frac{\delta^n}{1-\delta}d(x_0,x_1) \text{ if } \Phi(u,v)=u+v,
  $$

  $$
  d(x_n,x^*)\leqslant \delta^n d(x_0,x_1)  \text{ if }    \Phi(u,v)=\max\{u,v\},
  $$

  $$
  d(x_n,x^*)\leqslant \frac{\delta^n}{(1-\delta^q)^{\frac{1}{q}}}d(x_0,x_1) \text{ if }   \Phi(u,v)=(u^q+v^q)^{\frac{1}{q}}, q>0,
  $$
where $\delta=\alpha+\beta$.
\end{corollary}
\begin{proof}
To apply~(\ref{ee5}), we need to establish the continuity of \( d \) for every triangle function. The continuity of \( d \) is well-known for the cases where \(\Phi(u, v) = u + v\) and \(\Phi(u, v) = \max\{u, v\}\). The continuity of \( d \) when \(\Phi(u, v) = (u^q + v^q)^{\frac{1}{q}}\), with \(q > 0\), follows from Theorem~\ref{t13} (see also Example~\ref{ex1}).
\end{proof}

\begin{corollary}\label{c36}
Theorem~\ref{t2} holds for $b$-metric spaces {with the coefficient $K$ } if $\delta K<1$, where $\delta=\alpha+\beta$.
\end{corollary}

\begin{proof}
It is clear that $\Phi(u,v)=K(u+v)$ satisfies condition~(\ref{ee1}) and it is continuous at $(0,0)$. One can see from the proof that Lemma~\ref{lem} holds even in a weaker form: it is sufficient that condition~(\ref{ee2}) is satisfied with $\alpha$ from inequality~(\ref{e31n0}) but not necessarily for all $0\leqslant \alpha <1$. In this connection, consider expression~(\ref{ee2}) (with $\alpha=\delta$) for the function $\Phi$:
\begin{equation}\label{e519}
K+K^2\delta+K^3\delta^{2}+\cdots+K^p\delta^{p-1}+K^{p}\delta^{p}
\end{equation}
$$
\leqslant K+K^2\delta+K^3\delta^{2}+\cdots+K^p\delta^{p-1}+K^{p+1}\delta^{p}.
$$
It is clear that this sum consists of $p+1$ terms of geometric progression with the common ratio $\delta K$ and the start value $K$. According to the formula for the sum of infinite geometric series sum~(\ref{e519}) is less than $K/(1-\delta K)=C(\delta)$ for every finite $p\in \NN^+$, which establishes inequality~(\ref{ee2}).
\end{proof}

\begin{remark}
Note that in the case $\Phi(u,v)=K(u+v)$ estimation~(\ref{ee5}) cannot be established since in general case $d$ is not continuous in $b$-metric spaces with $K>1$, see~\cite{LHD19}.
\end{remark}

\begin{remark} It is interesting to note that despite the fact that conditions~(\ref{pc}) and~(\ref{pc1}) are equivalent replacing condition~(\ref{pc}) by condition~(\ref{pc1}) in Theorem~\ref{t2} requires an additional condition $\Phi(0, v) < 1$ for all $0 \leq v <1$. Indeed, let $x_0\in X$. Define $x_n = Tx_{n-1}$, i.e, $x_n=T^n x_0$ for $n = 1, 2, \ldots$. Similarly, in view of (\ref{pc1}), we get
\begin{align*}
    d(x_{n},x_{n+{1}})=d(Tx_{n-1},Tx_{n}) &\leqslant  \alpha d(x_{n-1}, x_{n}) + \beta d(x_{n},Tx_{n})  \\
               &=  \alpha d(x_{n-1}, x_{n}) +\beta d(x_{n}, x_{n+1}).
\end{align*}
Hence,
$
d(x_{n},x_{n+{1}}) \leqslant \delta d(x_{n-1},x_{n}),
$
where $\delta = \frac{\alpha}{1-\beta}$, $0\leqslant \delta <1$.

Also, by the generalized triangle inequality~(\ref{tr}), the monotonicity of $\Phi$ and (\ref{pc1}), we get
$$
d(x^*,Tx^*)\leqslant \Phi(d(x^*,T^nx_{0}),d(T^nx_{0},Tx^*))
$$
$$
\leqslant \Phi(d(x^*,T^nx_{0}),\alpha d(T^{n-1}x_{0}, x^*)+ \beta d(x^*,Tx^*)).
$$
Letting $n\to \infty$, by the continuity of $\Phi$, we obtain
$
d(x^*,Tx^*)\leqslant \Phi(0,\beta d(x^*,Tx^*)).
$
Using~(\ref{ee1}), we have
$
d(x^*,Tx^*)\leqslant d(x^*,Tx^*)\Phi(0,\beta).
$
By $\Phi(0,\beta)<1$, we get $d(x^*,Tx^*)=0$, i.e. $x^*=Tx^*$.
\end{remark}

\subsection{Weak contractions in semimetric spaces}

In 2004, the concept of a weak contraction \cite{B04} was introduced by Berinde within the framework of a metric space. Let $(X,d)$ be a metric space. A self-map $T\colon X \to X$ is a weak contraction on $X$ if there exist $\alpha, \delta \in \mathbb{R}$ with $0 \leq \alpha < 1$ and $\delta \geq 0$ such that for any $x, y \in X$,
\begin{equation} \label{wc}
d(Tx, Ty) \leq \alpha d(x, y) + \delta d(x, Ty).
 \end{equation}
Because of the symmetry property in a metric space,
$T$ also satisfies the following dual condition:
\begin{equation} \label{wc1}
d(Tx, Ty) \leq \alpha d(x, y) + \delta d(y, Tx).
\end{equation}
If $(X,d)$ is a metric space and $T \colon X \to X$ is either a Banach contraction, Kannan contraction, Chatterjea contraction or Zamfirescu contraction, then $T$ is a weak contraction \cite{B04}.\\

We now proceed to prove the following result:

\begin{theorem}\label{wt2}
Let $(X,d)$ be a complete semimetric space with the triangle function $\Phi$ continuous at $(0,0)$ and satisfying conditions~(\ref{ee1}) and~(\ref{ee2}).
Let $T \colon X \to X$ satisfy~(\ref{wc1}). Then $T$ has a fixed point. If $\delta=0$, then the fixed point is unique.
\end{theorem}
\begin{proof}
Let $x_0\in X$. Define $x_n = Tx_{n-1}$, i.e, $x_n=T^n x_0$ for $n = 1, 2, \ldots$. Then, in view of (\ref{wc1}), it follows
\begin{align*}
    d(x_{n},x_{n+{1}})=d(Tx_{n-1},Tx_{n}) &\leqslant  \alpha d(x_{n-1}, x_{n}) + \delta d(x_{n},Tx_{n-1})  \\
               &= \alpha d(x_{n-1}, x_{n}).
\end{align*}
Hence,
$$
d(x_{n},x_{n+{1}}) \leqslant \alpha d(x_{n-1},x_{n}),
$$
where $0 \leq \alpha <1$. By Lemma~\ref{lem}, $(x_n)$ is a Cauchy sequence and  by the completeness of $(X,d)$, this sequence has a limit $x^*\in X$. Let us prove that $Tx^*=x^*$. By the generalized triangle inequality~(\ref{tr}), the continuity of $\Phi$ and (\ref{wc1}), we have
\begin{align*}
d(x^*, Tx^*) &\leqslant \Phi(d(x^*, T^n x_0), d(T^n x_0, Tx^*)) \\
&= \Phi\left(d(x^*, T^n x_0), \alpha d(T^{n-1} x_0, x^*) + \delta d(x^*, T^n x_0)\right) \to 0,
\end{align*}
as $n \to \infty$, which means that $x^*$ is a fixed point. Let $\delta=0$. Suppose that there exist two distinct fixed points $x$ and $y$. Then $Tx=x$ and $Ty=y$, which contradicts~(\ref{wc1}), since $\alpha<1$.
\end{proof}

\begin{corollary}\label{c334}
Theorem~\ref{wt2} holds for semimetric spaces with the following triangle functions: $\Phi(u,v)=u+v$; $\Phi(u,v)=\max\{u,v\}$; $\Phi(u,v)=(u^q+v^q)^{\frac{1}{q}}$, $q>0$, and with the corresponding estimations from above for $d(x_n,x^*)$, see Corollary~\ref{c23}.
\end{corollary}

\begin{corollary}\label{c48}
Theorem~\ref{wt2} holds for $b$-metric spaces {with the coefficient $K$ } if $\alpha K<1$.
\end{corollary}
\begin{proof}
See the proof of Corollary~\ref{c36}.
\end{proof}

\begin{remark} It is noteworthy that although conditions~(\ref{wc1}) and~(\ref{wc}) are equivalent, substituting condition~(\ref{wc1}) with condition~(\ref{wc}) in Theorem~\ref{wt2} under the constraint $\alpha + 2\delta < 1$ necessitates additional requirements: (i) $\Phi(a, b) \leqslant a + b$ for all $a, b \geqslant 0$; (ii) $\Phi$ is continuous at the domain of definition.

Indeed, let $x_0\in X$. Define $x_n = Tx_{n-1}$, i.e, $x_n=T^n x_0$ for $n = 1, 2, \ldots$. Similarly, in view of (\ref{wc}), we get
\begin{align*}
    d(x_{n},x_{n+{1}})=d(Tx_{n-1},Tx_{n}) &\leqslant  \alpha d(x_{n-1}, x_{n}) + \delta d(x_{n-1},Tx_{n})  \\
               &=  \alpha d(x_{n-1}, x_{n}) +\delta d(x_{n-1}, x_{n+1}).
\end{align*}
Hence, by the generalized triangle inequality~(\ref{tr}) and condition (i) we get
\begin{align*}
d(x_{n}, x_{n+1}) & \leqslant \alpha d(x_{n-1}, x_{n}) + \delta \Phi(d(x_{n-1}, x_{n}), d(x_{n}, x_{n+1})) \\
& \leqslant \alpha d(x_{n-1}, x_{n}) + \delta (d(x_{n-1}, x_{n}) + d(x_{n}, x_{n+1})).
\end{align*}
Consequently,
\[
d(x_{n},x_{n+1}) \leqslant \gamma d(x_{n-1},x_{n}),
\]
where $\gamma = \frac{\alpha+\delta}{{1}-\delta}$, $0\leqslant \gamma <1$. By Lemma~\ref{lem}, $(x_n)$ is a Cauchy sequence and  by the completeness of $(X,d)$, this sequence has a limit $x^*\in X$.

Let us prove that $Tx^*=x^*$. 
By the generalized triangle inequality~(\ref{tr}), the monotonicity of $\Phi$ and~(\ref{wc}), we get
$$
d(x^*,Tx^*)\leqslant \Phi(d(x^*,T^nx_{0}),d(T^nx_{0},Tx^*))
$$
$$
= \Phi(d(x^*,T^nx_{0}), \alpha d(T^{n-1}x_{0}, x^*) + \delta (d(T^{n-1}x_{0},Tx^*))).
$$
Letting $n\to \infty$, by the continuity of $\Phi$, we obtain
$$
d(x^*,Tx^*)\leqslant \Phi(0,\delta d(x^*,Tx^*)).
$$
By condition (i), we get
$$
d(x^*,Tx^*)\leqslant \delta d(x^*,Tx^*).
$$
Hence, since $\delta<\frac{1}{2}$ we obtain the equality  $d(x^*,Tx^*)=0$.
\end{remark}

\subsection{Bianchini's contractions in semimetric spaces}

In 1972, Bianchini \cite{Bi72} proved the following result, which extends the Kannan fixed point theorem \cite{Ka68}. Let $T\colon X\to X$ be a mapping on a complete metric space $(X,d)$ such that
  \begin{equation}\label{kk}
   d(Tx,Ty)\leqslant \beta \max\{d(x,Tx), d(y,Ty)\},
  \end{equation}
where $0\leqslant \beta<1$ and $x,y \in X$. Then $T$ has a unique fixed point.

\begin{theorem}\label{t7}
Let $(X,d)$ be a complete semimetric space with the continuous triangle function $\Phi$, satisfying conditions~(\ref{ee1}) and~(\ref{ee2}) and additionally to the condition:
\begin{itemize}
\item [(i)]  $\Phi(0,v)<1$ \, for all \,  $0\leqslant v< 1$.
\end{itemize}
Let  $T\colon X\to X$ satisfy inequality~(\ref{kk}) with $0\leqslant \beta<1$. Then $T$ has a unique fixed point.
\end{theorem}

\begin{proof}
Let $x_0\in X$. Define $x_n = Tx_{n-1},$ i.e., $x_n=T^n x_0$ for $n = 1, 2, \ldots$. It follows straightforwardly that
$$
d(x_n,x_{n+1}) = d(Tx_{n-1},Tx_n)
$$
$$
\leqslant \beta \max\{d(x_{n-1},Tx_{n-1}), d(x_n,Tx_n)\} = \beta \max\{d(x_{n-1},x_n),  d(x_n,x_{n+1})\}.
$$

We now consider two cases:\\
Case 1:  $\max\{d(x_{n-1},x_n),  d(x_n,x_{n+1})\} = d(x_{n-1},x_n)$. Then we have
\begin{equation}\label{er1}
d(x_n,x_{n+1}) \leqslant \beta d(x_{n-1},x_n).
\end{equation}
Case 2:  $\max\{d(x_{n-1},x_n),  d(x_n,x_{n+1})\} = d(x_{n},x_{n+1})$. Then we get
\[
d(x_n,x_{n+1}) \leqslant \beta d(x_{n},x_{n+1})< d(x_n,x_{n+1}),
\]
which is a contradiction.

Hence, inequality~(\ref{er1}) holds for $0\leqslant \beta <1$. By Lemma~\ref{lem}, $(x_n)$ is a Cauchy sequence and  by the completeness of $(X,d)$, this sequence has a limit $x^*\in X$.

Let us prove that $Tx^*=x^*$. By the generalized triangle inequality~(\ref{tr}), the monotonicity of $\Phi$ and~(\ref{kk}), we get
$$
d(x^*,Tx^*)\leqslant \Phi(d(x^*,T^nx_{0}),d(T^nx_{0},Tx^*))
$$
$$
=\Phi(d(x^*,T^nx_{0}),\beta \max\{d(T^{n-1}x_{0},T^{n}x_{0}), d(x^*,Tx^*)\}).
$$
Letting $n\to \infty$, by the continuity of $\Phi$,  we obtain
$$
d(x^*,Tx^*)\leqslant \Phi(0,\beta d(x^*,Tx^*)).
$$
Using~(\ref{ee1}), we have
$$
d(x^*,Tx^*)\leqslant d(x^*,Tx^*)\Phi(0,\beta).
$$
By condition (i), we get $d(x^*,Tx^*)=0.$ Suppose that there exist two distinct fixed points $x$ and $y$. Then $Tx=x$ and $Ty=y$, which contradicts to~(\ref{kk}).
\end{proof}

\begin{corollary}
Theorem~\ref{t7} holds for semimetric spaces with the following triangle functions: $\Phi(u,v)=u+v$; $\Phi(u,v)=\max\{u,v\}$; $\Phi(u,v)=(u^q+v^q)^{\frac{1}{q}}$, $q>0$, and with the corresponding estimations from above for $d(x_n,x^*)$, see Corollary~\ref{c23}.
\end{corollary}

\begin{remark}
Note that the triangle function $\Phi(u,v)=K(u+v)$ does not satisfy condition (i) of Theorem~\ref{t7} in the case $K>1$.
\end{remark}

\subsection{Chatterjea-Bianchini contractions in semimetric spaces}

In~\cite{Ch72} Chatterjea  proved the following result. Let $T\colon X\to X$ be a mapping on a complete metric space $(X,d)$ such that
  \begin{equation}\label{Ch}
   d(Tx,Ty)\leqslant \beta (d(x,Ty)+d(y,Tx)),
  \end{equation}
where $0\leqslant \beta<\frac{1}{2}$ and $x,y \in X$. Then $T$ has a unique fixed point. If we substitute condition (\ref{Ch}) with the following
\begin{equation}\label{gCh}
   d(Tx,Ty)\leqslant \beta \max\{d(x,Ty), d(y,Tx)\},  0\leq \beta<1,
\end{equation}
we refer to it as Chatterjea-Bianchini contraction, since this condition has features of both Chatterjea and Bianchini contractions. Note that condition~(\ref{gCh}) appeared earlier in the  well-known survey paper of Rhoades \cite{Rh77} under number (12).

To prove the following theorem, we require the notion of an inverse function for a non-increasing function. This is necessary because the theorem aims to encompass the class of ultrametric spaces, and the function $\Psi(u) = \max\{1,u\}$ is not strictly increasing. By~\cite[p.~34]{GRSY12} for every non-decreasing function $\Psi\colon [0,\infty]\to [0,\infty]$ the inverse function $\Psi^{-1}\colon [0,\infty]\to [0,\infty]$ can be well defined by setting
$$
\Psi^{-1}(\tau)=\inf\limits_{\Psi(t)\geqslant \tau} t.
$$

Here, $\inf$ is equal to $\infty$ if the set of $t \in[0,\infty]$ such that $\Psi(t)\geqslant \tau$ is empty. Note that the function  $\Psi^{-1}$ is non-decreasing too. It is evident immediately by the definition that
\begin{equation}\label{r1}
  \Psi^{-1}(\Psi(t))\leqslant t \, \text{ for all } \, t\in [0,\infty].
\end{equation}

The proof of the following theorem is similar to the fixed point theorem for Chatterjea contraction mappings \cite{Ch72} in semimetric spaces with triangle functions \cite{PSB24}. For completeness, we provide the proof of this theorem.

\begin{theorem}\label{t08}
Let $(X,d)$ be a complete semimetric space with the continuous triangle function $\Phi$ such that the semimetric $d$ is continuous.
Let  $T\colon X\to X$ satisfy inequality~(\ref{gCh}) with some $\beta\geqslant0$.
Let $\Phi$ also satisfy conditions~(\ref{ee1}) and~(\ref{ee2}) and additionally the following conditions hold:
\begin{itemize}
  \item [(i)] $\Phi(0,\beta)<1$,
  \item [(ii)] $\Psi^{-1}(1/\beta)> 1$ if $\beta>0$, where $\Psi(u)=\Phi(u,1)$.
\end{itemize}
Then $T$ has a fixed point. If $0\leqslant\beta<1$, then the fixed point is unique.
\end{theorem}

\begin{proof}
Let $\beta=0$. Then~(\ref{gCh}) is equivalent to $d(Tx,Ty)=0$ for all $x,y \in X$. Let $x_0\in X$ and $x^*=Tx_0$, then $d(Tx_0,T(Tx_0))=0$ and $d(x^*,Tx^*)$=0. Hence, $x^*$ is a fixed point. Suppose there exists another fixed point $x^{**}\neq x^*$, $x^{**}=Tx^{**}$. Then, by the equality $d(Tx,Ty)=0$, we have $d(Tx^*,Tx^{**})=d(x^*,x^{**})=0$, which leads to a contradiction.\

Let $\beta>0$ and $x_0 \in X$ be given. Define $x_n =Tx_{n-1} $, i.e., $x_n =T^n x_0$ for $n = 1, 2, \ldots$. If $x_i=x_{i+1}$ for some $i$, then it is clear that $x_i$ is a fixed point. Suppose that $x_i\neq x_{i+1}$ for all $i$. It is evident that
\begin{align*}
    d(x_{n}, x_{n+1}) &= d(Tx_{n-1}, Tx_{n}) \leq \beta \max\{d(x_{n-1}, Tx_{n}), d(x_{n}, Tx_{n-1})\} \\
    &= \beta \max\{d(x_{n-1}, x_{n+1}), d(x_{n}, x_{n})\} \\
    &= \beta \max\{d(x_{n-1}, x_{n+1}), 0\} \\
    &= \beta d(x_{n-1}, x_{n+1}).
\end{align*}
Hence, by the generalized triangle inequality~(\ref{tr}) and condition~(\ref{ee1}), we obtain
$$
d(x_{n},\,x_{{n+1}})\leqslant \beta \Phi(d(x_{{n-1}},\,x_{{n}}), d(x_{{n}},\,x_{{n+1}}))
$$
and
\begin{equation}\label{w1}
\frac{1}{\beta}\leqslant \Phi\left(\frac{d(x_{{n-1}},\,x_{{n}})}{d(x_{n},\,x_{{n+1}})}, 1\right)=\Psi\left(\frac{d(x_{{n-1}},\,x_{{n}})}{d(x_{n},\,x_{{n+1}})}\right),
\end{equation}
where $\Psi(u)=\Phi(u,1)$, $u\in [0,\infty)$. It is clear that $\Psi(u)$ is non-decreasing on $[0,\infty)$. Consequently, $\Psi^{-1}(u)$ is also non-decreasing on $[0,\infty)$.

Hence, it follows from~(\ref{r1}) and~(\ref{w1}) that
$$
\Psi^{-1}\bigg(\frac{1}{\beta}\bigg)\leqslant \frac{d(x_{{n-1}},\,x_{{n}})}{d(x_{n},\,x_{{n+1}})}
$$
and
$$
d(x_{n},\,x_{{n+1}}) \leqslant \Big(\Psi^{-1}\Big({1}/{\beta}\Big)\Big)^{-1} {d(x_{{n-1}},\,x_{{n}})}.
$$
Consequently,
\[
d(x_{n},x_{n+1}) \leqslant \alpha d(x_{n-1},x_{n}),
\]
where $\alpha = \Big(\Psi^{-1}\Big({1}/{\beta}\Big)\Big)^{-1}$.
Since by condition (ii) $\Psi^{-1}(1/\beta)> 1$, we get  $0\leqslant \alpha <1$. By Lemma~\ref{lem}, $(x_n)$ is a Cauchy sequence and  by the completeness of $(X,d)$, this sequence has a limit $x^*\in X$. Let us prove that $Tx^*=x^*$.
By the generalized triangle inequality~(\ref{tr}), the monotonicity of $\Phi$ and~(\ref{gCh}), we get
$$
d(x^*,Tx^*)\leqslant \Phi(d(x^*,T^nx_{0}),d(T^nx_{0},Tx^*))
$$
$$
\leqslant \Phi(d(x^*,T^nx_{0}),\beta \max(d(T^{n-1}x_{0},Tx^*), d(x^*,T^nx_0))).
$$
Letting $n\to \infty$, by the continuity of $\Phi$ and $d$ we obtain
$$
d(x^*,Tx^*)\leqslant \Phi(0,\beta d(x^*,Tx^*)).
$$
Using~(\ref{ee1}), we have
$$
d(x^*,Tx^*)\leqslant d(x^*,Tx^*)\Phi(0,\beta).
$$
By condition (i), we get $d(x^*,Tx^*)=0$.

Suppose that there exist two distinct fixed points $x$ and $y$. Then $Tx=x$ and $Ty=y$, which contradicts to~(\ref{gCh}) in the case $0\leqslant \beta<1$.
\end{proof}


\begin{corollary}\label{c45}
Theorem~\ref{t08} holds in ultrametric spaces with the coefficient $0\leqslant \beta < 1$.
\end{corollary}
\begin{proof}
According to the assumption $\Phi(u,v)=\max\{u,v\}$, $\Psi(u)=\max\{u,1\}$ and
$$
\Psi^{-1}(u)=
  \begin{cases}
    0, & u\in[0,1], \\
    u, & u\in(1,\infty).
  \end{cases}
$$
Clearly, condition (i) holds for all $0\leqslant\beta<1$ and condition (ii) holds for all $0<\beta<1$.
\end{proof}

\begin{corollary}\label{c46}
Theorem~\ref{t08} holds for semimetric spaces with the following triangle functions $\Phi(u,v)=(u^q+v^q)^{\frac{1}{q}}$, $q>0$, and with the coefficient $0\leqslant\beta<2^{-1/q}$ in~(\ref{gCh}).
\end{corollary}
\begin{proof}
By Example~\ref{ex1} the semimetric $d$ is continuous. Further, we have  $\Psi(u)=(u^q+1)^{\frac{1}{q}}$ and $\Psi^{-1}(u)=(u^q-1)^{\frac{1}{q}}$. Clearly, condition (i) holds for all $0\leqslant\beta<1$ but condition (ii) holds if $0<\beta<2^{-1/q}$.
\end{proof}

\section{Conclusion}

This paper extends several known general fixed point theorems related to weak and partial contractions by modifying assumptions. Using Lemma~\ref{lem}, the study deepens the understanding of fixed point theorems and broadens their applicability beyond metric spaces. The findings suggest potential future research directions, including applying the approach to other contractive mappings under diverse conditions and categorizing various contractive definitions that can be proved using this technique. The significance of these generalized theorems extends across multiple disciplines, including optimization, mathematical modeling, and computer science. They may serve to establish stability conditions, demonstrate the existence of optimal solutions, and improve algorithm design.

\textbf{Acknowledgements.} This work was partially supported by a grant from the Simons Foundation (PD-Ukraine-00010584, E. Petrov). 

\bigskip

CONTACT INFORMATION

\medskip

Ravindra Kishor Bisht \\
Department of Mathematics, National Defence Academy, Khadakwasla, Pune, India\\
E-Mail: ravindra.bisht@yahoo.com

\medskip

Evgen Oleksandrovych Petrov \\
Institute of Applied Mathematics and Mechanics of the NAS of Ukraine, Sloviansk, Ukraine \\
E-Mail: eugeniy.petrov@gmail.com
\end{document}